\documentclass[12pt,letterpaper]{article}

\usepackage{amssymb,amsfonts,amscd,amsthm}
\usepackage[all,arc]{xy}
\usepackage{enumerate, bbm}
\usepackage{mathrsfs}
\usepackage{tikz}
\usepackage[left=0.9in,top=0.9in,right=0.9in,bottom=0.9in]{geometry}
\usepackage{mathtools}
\usepackage{hyperref}
\usepackage{color}%colors
\usepackage{graphicx}
\usepackage{enumitem} %Allows resuming enumeration and numbering by section
\graphicspath{ {TexImages/} }
\usepackage{cleveref}

\newtheorem{thm}{Theorem}[section]
\newtheorem{cor}[thm]{Corollary}
\newtheorem{prop}[thm]{Proposition}
\newtheorem{lem}[thm]{Lemma}
\newtheorem{claim}[thm]{Claim}

\newtheorem{prob}[thm]{Problem}

\newtheorem{conj}[thm]{Conjecture}

\theoremstyle{definition}
\newtheorem{defn}{Definition}

\hypersetup{
	colorlinks,
	citecolor=blue,
	filecolor=blue,
	linkcolor=blue,
	urlcolor=blue,
	linktocpage
}

\setlist[enumerate]{itemsep=2ex, topsep=2ex} %spaces out enumerate/itemize better
\setlist[itemize]{itemsep=2ex, topsep=2ex}

\newcommand{\F}{\mathbb{F}}

\newcommand{\E}{\mathbb{E}}

\newcommand{\al}{\alpha}

\newcommand{\ep}{\varepsilon}

\newcommand{\Om}{\Omega}

\newcommand{\del}{\delta}

\renewcommand{\l}{\left}
\renewcommand{\r}{\right}

\newcommand{\half}{\frac{1}{2}}

\newcommand{\sm}{\setminus}
\newcommand{\sub}{\subseteq}

\renewcommand{\c}[1]{\mathcal{#1}}
\renewcommand{\b}[1]{\mathbf{#1}}

\newcommand{\ol}[1]{\overline{#1}}

\newcommand{\f}[2]{\frac{#1}{#2}}

\newcommand{\mr}[1]{\mathrm{#1}}

\renewcommand{\SS}[1]{\textcolor{red}{#1}}

\newcommand{\BG}[1]{\textcolor{blue}{#1}}

\newcommand{\ex}{\mr{ex}}

\newcommand{\BB}{\mathcal{B}}

\newcommand{\fff}{\beta_r(F)}

\newcommand{\N}{\c{N}} %I think this is the most standard way to write things

\title{Clique Supersaturation}
\author{Quentin Dubroff\footnote{Dept.\ of Mathematics, Rutgers University {\tt qcd2@math.rutgers.edu}}  \and Benjamin Gunby\footnote{Dept.\ of Mathematics, Rutgers University {\tt bg570@rutgers.edu}} \and Bhargav Narayanan\footnote{Dept.\ of Mathematics, Rutgers University {\tt narayanan@math.rutgers.edu}. This research was supported by NSF grants DMS-2237138 and CCF-1814409, and a Sloan Research Fellowship.} \and  Sam Spiro\footnote{Dept.\ of Mathematics, Rutgers University {\tt sas703@scarletmail.rutgers.edu}. This material is based upon work supported by the National Science Foundation Mathematical Sciences Postdoctoral Research Fellowship under Grant No. DMS-2202730.}}
\date{\today}
\begin{document}
	\maketitle
	\begin{abstract}
		We study how many copies of a graph $F$ that another graph $G$ with a given number of cliques is guaranteed to have.  For example, one of our main results states that for all $t\ge 2$, if $G$ is an $n$ vertex graph with $kn^{3/2}$ triangles and $k$ is sufficiently large in terms of $t$, then $G$ contains at least 
		\[\Omega(\min\{k^t n^{3/2},k^{\frac{2t^2}{3t-1}}n^{\frac{5t-2}{3t-1}}\})\]
		copies of $K_{2,t}$, and furthermore, we show these bounds are essentially best-possible provided either $k\ge n^{1/2t}$ or if certain bipartite-analogues of well known conjectures for Tur\'an numbers hold.
	\end{abstract}
	% Moreover, we show that for $t$ sufficiently large there exist $n$ vertex graphs with $\Omega(kn^{3/2})$ triangles and with at most\[\min\{k^t n^{3/2+o(1)},O(k^{\frac{2t^2}{3t-1}}n^{\frac{5t-2}{3t-1}})\}\]copies of $K_{2,t}$.  
	\section{Introduction}
	In this paper, we shall study a generalised supersaturation problem. Broadly speaking, `extremal problems' ask for the largest `size' $N$ that a combinatorial object can have before containing at least one structure $F$, and in turn, `supersaturation problems' ask about how many copies of $F$ are guaranteed to exist if a combinatorial object has `size' substantially larger than the extremal value $N$. In addition to being natural refinements of extremal problems in their own right, supersaturation problems also arise often in a number of other contexts. For example, supersaturation results were used by Erd\H{o}s and Simonovits~\cite{erdos1984cube} to obtain upper bounds on the Tur\'an number of the hypercube.  More recently, supersaturation has proven to be a key ingredient for various asymptotic enumeration results proved using the method of hypergraph containers developed independently by Balogh, Morris, and Samotij~\cite{balogh2015independent} and by Saxton and Thomasson~\cite{saxton2015hypergraph}. 
	
	Here, we investigate the following problem: how many copies of a graph $F$ can we guarantee in a graph $G$ with a specified number of copies of another graph $H$?  More precisely, given two graphs $F$ and $G$, we define
	\[\N(F,G)=\#\,\textrm{subgraphs of }G\textrm{ isomorphic to }F,\]
	and our aim then is to prove lower bounds on $\N(F,G)$ as a function of $\N(H,G)$.  We will informally refer to problems of this form as \textit{generalized supersaturation problems}.
	
	An immediate obstruction to this problem is the existence of $F$-free graphs $G$ which contain a large number of copies of $H$. To this end, we define the \textit{generalized Tur\'an number},  first introduced by Alon and Shikhelman~\cite{alon2016many}, by
	\[\ex(n,F,H)=\max\{ \N(H,G): G\textrm{ is an }F\textrm{-free graph on }n\textrm{ vertices}\},\]
	and we write $\ex(n,F):=\ex(n,F,K_2)$ to denote the (classical) \textit{Tur\'an number of $F$}.  Observe that, by definition, if $G$ is an $n$ vertex graph then
	\begin{equation}\N(H,G)>\ex(n,F,H)\implies \N(F,G)>0,\label{eq:triv}\end{equation}
	and this is best possible since there exist $F$-free graphs with $\N(H,G)\le \ex(n,F,H)$.
	
	We are interested in quantitative versions of the trivial bound \eqref{eq:triv}.  For example, Halfpap and Palmer~\cite{halfpap2021supersaturation} proved that if $\chi(F)>\chi(H)$ and $\ep>0$, then any $n$ vertex graph $G$ with $\N(H,G)\ge (1+\ep) \ex(n,F,H)$ has $\N(F,G)=\Om_\ep(n^{v(F)})$, i.e., $G$ contains a constant proportion of the copies of $F$ in $K_n$.  Hence, the central interest in the case $\chi(F)>\chi(H)$ is in proving asymptotically tight bounds for $\N(F,G)$ as a function of $\ep$. For example, work of Razborov~\cite{razborov2008minimal} completely solves this asymptotic problem of minimizing the number of $K_3$'s in a graph with a given number of edges, and this was later generalized by Reiher~\cite{reiher2016clique} to handle general cliques $K_r$ in place of the triangle $K_3$.
	
	In this paper, we focus on the `degenerate' setting $\chi(F)\le \chi(H)$ where the focus for supersaturation centers around proving coarse (i.e., order of magnitude) bounds on $\N(F,G)$.  One classical example in this setting is the following conjecture of Erd\H{o}s and Simonovits~\cite{erdos1984cube}.
	\begin{conj}[\cite{erdos1984cube}]\label{conj:ErdosSimonovits}
		Let $F$ be a graph with $\ex(n,F)=O(n^\al)$.   If $G$ is an $n$-vertex graph with $e(G)=kn^\al$ and $k\ge k_0(F)$, then
		\[\c{N}(F,G)=\Om(k^{e(F)} n^{v(F)-(2-\al)e(F)}).\]
	\end{conj}
	We note that this conjecture, if true, would be best possible by considering $G$ to be the random graph with $kn^\al$ edges.  \Cref{conj:ErdosSimonovits} is known to hold (possibly with non-optimal values of $\al$) for a large number of graphs, such as even cycles~\cite{morris2016number} and all graphs which satisfy Sidorenko's conjecture~\cite[Theorem 9]{sidorenko1991inequalities}.  One result of particular importance to us will be the following which confirms \Cref{conj:ErdosSimonovits} for complete bipartite graphs when $\al=2-1/s$. 
	\begin{prop}[\cite{erdHos1983supersaturated}]\label{prop:ESKst}
		For all $s\le t$, there exists a constant $C=C(s,t)$ such that if $G$ is an $n$-vertex graph with $e(G)=kn^{2-1/s}$ and $k\ge C$, then $\N(K_{s,t},G)=\Om(k^{st}n^s)$.
	\end{prop}
	
	Other generalized supersaturation results in the degenerate setting include work of Cutler, Nir, and Radcliffe~\cite{cutler2022supersaturation} who studied the case when $F,H$ are each either cliques or stars; as well as Gerbner, Nagy, and Vizer~\cite{gerbner2022unified} who initiated the systematic study of generalized supersaturation results and who proved a number of results when $F,H$ are both bipartite.

	\subsection{Our results}
	In this paper, we focus on generalized supersaturation problems when $H=K_r$.  That is, we ask how many copies of a given graph $F$ is guaranteed in another graph $G$ if $G$ has $N$ copies of $K_r$, and we informally refer to this as the \textit{clique supersaturation problem}. For example, we have the following.
	\begin{lem}\label{lem:Kruskal}
		If $F$ is a graph with $v(F)\le r$ and if $G$ is a graph with $\N(K_r,G)=N$, then $\N(F,G)=\Om(N^{v(F)/r})$.  Moreover, the graph $G$ consisting of a clique of size $N^{1/r}$ satisfies $\N(K_r,G)=\Om(N)$ and $\N(F,G)=O(N^{v(F)/r})$.
	\end{lem}
	\Cref{lem:Kruskal} follows immediately from the Kruskal-Katona theorem, which implies that any graph with $N$ copies of $K_r$ has at least $\Om(N^{v(F)/r})$ copies of $K_{v(F)}$. Due to \Cref{lem:Kruskal}, we will only consider $F$ with $v(F)>r$ throughout this paper.
	
	Returning to the general problem: when $r=2$, \Cref{conj:ErdosSimonovits} predicts that the solution to the clique supersaturation problem is always achieved by the random graph $G_{n,p}$ with $N$ copies of $K_2$, i.e.\ when $p=N n^{-2}$.  For larger $r$, it again makes sense to look at what happens for $G_{n,p}$.  To this end, if we want $G_{n,p}$ to have on the order of $N$ copies of $K_r$, then we should take $p=(N n^{-r})^{1/{r\choose 2}}$, which gives
	\begin{equation}\E[\N(F,G_{n,p})]=\Theta\left((N n^{-r})^{e(F)/{r\choose 2}}n^{v(F)}\right).\label{eq:Gnp}\end{equation}
	Our first main result significantly improves upon this trivial bound for a wide range of $F$ and $r\ge 3$. For this result, we recall that a graph $F$ is \textit{2-balanced} if for all $F'\sub F$ with $v(F')\ge 3$, we have \[\frac{e(F')-1}{v(F')-2}\le \frac{e(F)-1}{v(F)-2}.\]
	
	\begin{thm}\label{thm:randomCliques}
		Let $F$ be a 2-balanced graph  with $e(F)\ge 2$ and let $2\le r< v(F)$ be an integer.  For all $1\le N\le {n\choose r}$, there exists an $n$-vertex graph $G$ with $\N(K_r,G)=\Om(N)$ and with
		\[\N(F,G)=O\left((N n^{-r})^{e(F)\fff} n^{v(F)}\right),\]
		where
		\[\fff:=\frac{v(F)-2}{(r-2)(e(F)-1)+v(F)-2}\]
	\end{thm}
	
	Note that this bound is strictly smaller than the bound \eqref{eq:Gnp} from $G_{n,p}$ whenever $\fff> {r\choose 2}^{-1}$,
	and this is equivalent to having 
	\[\frac{r+1}{2}>\frac{e(F)-1}{v(F)-2}\hspace{1em}\mathrm{for\ }r,v(F)\ge 3.\]
	For example, this inequality holds if $F=K_{2,t}$ when $t\ge 2$ and $r\ge 3$. More generally, \Cref{thm:randomCliques} implies the same result holds when $K_r$ is replaced by any $r$-vertex graph $H$.  In this case, the bound does better than the corresponding bound coming from $G_{n,p}$ precisely when
	\[\frac{e(H)-1}{v(H)-2}>\frac{e(F)-1}{v(F)-2}\]
	provided $v(H),v(F)\ge 3$.
	
	A crucial part of our proof of \Cref{thm:randomCliques} will be the following general lower bound on $\ex(n,K_r,F)$, which may be of independent interest.
	\begin{thm}\label{thm:TuranBound}
		If $F$ is a 2-balanced graph with $e(F)\ge 2$, then for all $2\le r< v(F)$ we have
		\[\ex(n,K_r,F)=\Om(n^{2-\frac{v(F)-2}{e(F)-1}}).\]
	\end{thm}
	\Cref{thm:TuranBound} recovers the classic result $\ex(n,F)=\Om(n^{2-\frac{v(F)-2}{e(F)-1}})$ for 2-balanced graphs, though we emphasize that the proof for $r>2$ is somewhat more involved than the easy deletion argument which proves the classic $r=2$ case.  We also note that \Cref{thm:TuranBound} can be close to best possible.  For instance, it is known that $\ex(n,K_r,K_{2,t})=\Theta_t(n^{3/2})$ for $t$ sufficiently large in terms of $r$ \cite{alon2016many, spiro2022random}, and in this case, \Cref{thm:TuranBound} gives a lower bound of $\Om(n^{3/2-\frac{1}{4t-2}})$, which is quite close to tight.
	
	%As \Cref{thm:TuranBound} is not the primary focus of the paper, we postpone discussing it in detail until Subsection~\ref{sub:Turan}.

	%It is tempting to think that \Cref{thm:TuranBound} could hold for all graphs $F$ after replacing $\frac{e(F)-1}{v(F)-2}$ with the 2-density $m_2(F):=\max_{F'\sub F} \frac{e(F')-1}{v(F')-2}$.  However, this is easily seen to be false by considering $F$ to be the union of $K_r$ together with some number of isolated vertices, as in this case $\ex(n,K_r,F)=0$ whenever $n\ge v(F)$.  Such a result does hold (and follows immediately from the statement of \Cref{thm:TuranBound}) provided the $F'\sub F$ achieving $m_2(F)$ has $v(F')>r$; and this condition is implicitly present in the classical $r=2$ bound of $\ex(n,F)=\Om(n^{2-1/m_2(F)})$ for graphs $F$ with $e(F)\ge 2$ (since $e(F)\ge 2$ implies $v(F')>2$).

	We next turn to bounds for specific choices of $F$.  For this, it will not make sense to consider an arbitrary number of cliques $N$, as no copies of $F$ will be guaranteed if $N\le \ex(n,K_r,F)$.  As such, we will normalize the $N$ in our results by replacing $N$ with $kn^\al$ whenever\footnote{We will specifically normalize our results based off the following bounds: (1) $\ex(n,K_r,T)=\Theta(n)$ whenever $T$ is a tree, (2) $\ex(n,K_r,K_{2,t})=O(n^{3/2})$, with this result being tight if $t$ is sufficiently large in terms of $r$, (3) $\ex(n,K_r,K_{s,t})=O(n^{r-{r\choose 2}/s})$ whenever $r\le s$, with this result being tight if $t$ is sufficiently large in terms of $r$; see  \cite{alon2016many, spiro2022random}.} $\ex(n,K_r,F)=O(n^\al)$.
	
	Perhaps the most natural case of $F$ to consider for the clique supersaturation problem is when $F=K_t$ is itself a clique.  The case $t\le r$ is completely solved by the Kruskal-Katona theorem, and the case $t>r$ is solved up to order of magnitude by the result of Halfpap and Palmer~\cite{halfpap2021supersaturation}.  
	
	After cliques, the next simplest case is when $F$ is a tree.  This too is relatively easy to solve. 
	\begin{prop}\label{prop:trees}
		For all trees $T$ and integers $2\le r<v(F)$, there exists a constant $k_0=k_0(T)$ such that if $G$ is an $n$-vertex graph with $\N(K_r,G)=k n$ and $k\ge k_0$, then
		\[\N(T,G)=\Om(k^{(v(T)-1)/(r-1)}n).\]
		Moreover, the graph $G$ consisting of the disjoint union of $k^{-1/(r-1)}n$ cliques of size $k^{1/(r-1)}$ satisfies  $\N(K_r,G)=\Om(kn)$ and $\N(T,G)=O(k^{(v(T)-1)/(r-1)}n)$
	\end{prop}
	Alternatively, \Cref{thm:randomCliques} can be used in \Cref{prop:trees} instead of the disjoint union of cliques to get the same bound.  Sharper bounds for $F=K_{1,s}$ were obtained by Cutler, Nir and Radcliffe~\cite[Theorem 1.9]{cutler2022supersaturation} when $\N(K_r,G)=(1+\ep) \ex(n,K_{1,s},K_r)$.  This and the Kruskal-Katona theorem are the only results we are aware of studying degenerate clique supersaturation problems prior to this work.

	We next look at complete bipartite graphs $K_{s,t}$.  This is a natural case to study given that \Cref{prop:ESKst} is relatively easy to prove and essentially solves the case of $r=2$.  However, this problem becomes significantly more complex for $r>2$, even in the simplest (non-tree) case of $s=2$. % To this end, we state the following bounds which are essentially tight whenever $t$ is sufficiently large.  

	\begin{thm}\label{thm:graphK2t}
		For all integers $t\ge 2$ and $2< r<2+t$, there exists a constant $k_0=k_0(t)$ such that if $G$ is an $n$-vertex graph with $\N(K_r,G)=k n^{3/2}$ and $k\ge k_0$, then
		\[\N(K_{2,t},G)=\Om(\min\{k^{\f{t}{r-2}}n^{3/2},k^{\f{2t^2}{(2t-1)(r-2)+t}}n^{\f{(3t-2)(r-2)+2t}{(2t-1)(r-2)+t}}\}).\]
		Moreover, if $2<r<2+t$, there exists an $n$ vertex graph $G$ with $\N(K_r,G)=\Om(kn^{3/2})$ and
		\[\N(K_{2,t},G)= O(k^{\f{2t^2}{(2t-1)(r-2)+t}}n^{\f{(3t-2)(r-2)+2t}{(2t-1)(r-2)+t}}).\]
		Also, conditional on the existence of a bipartite graph $B$ on $U\cup V$ with $|U|=k^{-2/(r-2)}n^{3/2},\ |V|=n,\ e(B)=\Om(k^{-1/(r-2)} n^{3/2})$, and such that every pair of vertices in $V$ has fewer than $t$ paths of length at most 4 between them; there exists an $n$ vertex graph $G$ with $\N(K_r,G)=\Om(kn^{3/2})$ and
		\[\N(K_{2,t},G)= O(\min\{k^{\f{t}{r-2}}n^{3/2},k^{\f{2t^2}{(2t-1)(r-2)+t}}n^{\f{(3t-2)(r-2)+2t}{(2t-1)(r-2)+t}}\}).\]
	\end{thm}
	The bipartite graph $B$ in the conditional portion of this theorem is closely related to a bipartite analogue of the infamous problem of determining $\ex(n,C_8)$; see the concluding remarks for more on this.

	One might hope that the methods of \Cref{thm:graphK2t} could be extended to $K_{s,t}$ in general, but there are fundamental obstacles to this.  Indeed, every construction in this paper will turn out to be the union of nearly-disjoint cliques of roughly the same size, and one can essentially show\footnote{Of course, every graph $G$ is the disjoint union of $K_2$'s, so this exact claim is not quite true.  However, \Cref{lem:KstFail} will imply that this claim is true if we further impose that a large portion of the $K_r$'s in $G$ lie within the cliques used in its union.} that any construction of this form will fail to do better than the random graph $G_{n,p}$ for $K_{s,t}$ when $r\le s$; see \Cref{lem:KstFail} for an exact statement and the concluding remarks for further discussions.

	\section{Supersaturation}
	In this section we prove our supersaturation results, i.e.\ the lower bounds of \Cref{prop:trees} and \Cref{thm:graphK2t}.  We begin with two preliminary results that will be useful in our proofs.  First, the following shows that if a  set of vertices is contained in many copies of $K_r$, then they must have many neighbors. 
	\begin{lem}\label{lem:commonNeighbors}
		If $G$ is a graph and $S\sub V(G)$ is a set of vertices which is contained in at least $\ell$ cliques of size $r>|S|$, then there are at least $\ell^{1/(r-|S|)}$ vertices adjacent to every vertex of $S$.
	\end{lem}
	\begin{proof}
		Let $N(S)$ denote the set of vertices adjacent to every vertex of $S$.  Observe that every $K_r$ containing $S$ can be identified by choosing $r-|S|$ vertices from $N(S)$.  Thus we must have \[\ell\le {|N(S)|\choose r-|S|}\le |N(S)|^{r-|S|},\]
		from which the result follows.
	\end{proof}
	
	To motivate our next lemma, we note that in proving \Cref{prop:trees}, it would be useful to work with a subgraph $G'\sub G$ which has large minimum degree. Unfortunately, the standard lemma saying that we can find a $G'\sub G$ of minimum degree comparable to the average degree of $G$ will be too weak for our purposes, as it could be the case that $G'$ has very few vertices (in which case we may not be able to find many copies of $T$ in $G'$).  We get around this with the following lemma from \cite[Lemma 2.5]{mckinley2023random}, which gives substantially stronger bounds on $\del(G')$ if $v(G')$ is small.  We will in fact need a slight generalization of this result to hypergraphs, which can be proven with an identical argument.
	
	\begin{lem}[\cite{mckinley2023random}]\label{lem:minDegree}
		Let $H$ be an $n$-vertex hypergraph with $\emptyset\notin E(H)$. For all real $b\ge 1$, there exists a subgraph $H'\sub H$ with $v(H')>0$ and minimum degree at least \[2^{-b}\left(\frac{v(H')}{n}\right)^{1/b}\f{e(H)}{ v(H')}.\]
	\end{lem}

	With this we can prove our supersaturation result for trees.
	
	\begin{proof}[Proof of \Cref{prop:trees}]
		Recall that we wish to show for all trees $T$ and $2\le r<v(T)$, if $T$ is a tree and $G$ is an $n$-vertex graph with $kn$ copies of $K_r$, then $G$ contains at least $\Om(k^{(v(T)-1)/(r-1)})$ copies of $T$ provided $k$ is sufficiently large in terms of $T$.  
		
		Let $H$ by the $r$-uniform hypergraph with $V(H)=V(G)$ whose hyperedges are copies of $K_r$ in $G$.  Let $H'\sub H$ be the subhypergraph guaranteed by \Cref{lem:minDegree} with $b=\frac{v(T)-1}{v(T)-r}>1$, and let $G'\sub G$ by the induced subgraph with $V(G')=V(H')$.  By unwinding the definitions, we see that every vertex of $G'$ is contained in at least
		\[\ell:=2^{-\frac{v(T)-1}{v(T)-r}}\left(\frac{v(G')}{n}\right)^{(v(T)-r)/(v(T)-1)}\f{kn}{ v(G')}=2^{-\frac{v(T)-1}{v(T)-r}} k n^{\f{r-1}{v(T)-1}} v(G')^{\f{1-r}{v(T)-1}}\]
		copies of $K_r$, which by \Cref{lem:commonNeighbors} implies every vertex of $G'$ has degree at least $\ell^{1/(r-1)}$.  Observe that $\ell\ge 2^{-\frac{v(T)-1}{v(T)-r}}k$, so by taking $k$ sufficiently large, we may assume $\ell^{1/(r-1)} \ge 2v(T)$.  
		
		With this in mind, we claim that
		\[\N(T,G')\ge v(G')\cdot (\ell^{1/(r-1)}-v(T))^{v(T)-1}/v(T)!.\]
		Indeed, because $T$ is a tree, we can order its vertices $x_1,\ldots,x_{v(T)}$ in such a way that every $x_i$ with $i>1$ has a (unique) neighbor $x_j$ with $j<i$.  With this ordering, we build our copies of $T$ greedily by selecting any vertex of $G'$ and label it $y_1$, and then iteratively given that we have chosen $y_1,\ldots,y_{i-1}$ and that $x_i$ is adjacent to $x_j$ with $j<i$, we choose $y_i$ to be any neighbor of $y_j$ that is not equal to any of the already selected vertices $y_1,\ldots,y_{i-1}$.  It is not difficult to check that this procedure terminates with a set of vertices $y_1,\ldots,y_{v(T)}$ which forms a copy of $T$ in $G'\sub G$, and that the number of ways of going through this procedure is at least $v(G')\cdot (\ell^{1/(r-1)}-v(T))^{v(T)-1}$. The same tree can be generated at most $v(T)!$ times by this algorithm, giving the bound above.
		
		Using $\ell^{1/(r-1)}\ge 2 v(T)$ and the definition of $\ell$, we obtain
		\[\N(T,G)\ge \N(T,G')=\Om(v(G')\ell^{(v(T)-1)/(r-1)})=\Om( k^{(v(T)-1)/(r-1)}n),\]
		giving the desired result.
	\end{proof}
	We next prove our supersaturation result for $K_{2,t}$.
	\begin{proof}[Proof of \Cref{thm:graphK2t} Lower Bound]
		Recall that we wish to show that for all $t\ge 2$ and $2\le r<2+t$ that if $G$ is an $n$-vertex graph with $kn^{3/2}$ copies of $K_r$, then $G$ contains at least \[\Om(\min\{k^{\f{t}{r-2}}n^{3/2},k^{\f{2t^2}{(2t-1)(r-2)+t}}n^{\f{(3t-2)(r-2)+2t}{(2t-1)(r-2)+t}}\})\] copies of $K_{2,t}$ provided $k$ is sufficiently large in terms of $t$.  
		
		The idea behind our proof is the following: either our graph $G$ has many edges (in which case it will contain many copies of $K_{2,t}$ by \Cref{prop:ESKst}), or we can assume every pair of adjacent vertices has many common neighbors (in which case we can build copies of $K_{2,t}$ greedily).  More precisely, \Cref{prop:ESKst} implies there exists a constant $C=C(t)$ such that
		\begin{equation}\textrm{If }e(G)\ge Cn^{3/2}\textrm{, then } \N(K_{2,t},G)=\Om(e(G)^{2t} n^{2-3t}).\label{eq:graphDense}\end{equation}
		We use the following when $e(G)$ is small.
		\begin{claim}
			If  $e(G)\le \max\{C,k^{1/2}\} n^{3/2}$ and $k$ is sufficiently large in terms of $C$, then 
			\begin{equation}\label{eq:expandTriangles}\N(K_{2,t},G)=\Omega\left(k^{\f{t}{r-2}}n^{\f{3t}{2(r-2)}}e(G)^{\f{r-2-t}{r-2}}\right).\end{equation}
		\end{claim}
		\begin{proof}
			We form copies of $K_{2,t}$ by starting with an edge $e=uv$ and then choosing any $t$ of the common neighbors of $u,v$. Note that this process generates each $K_{2,t}$ in at most 2 ways (and in at most 1 way if $t>2$).
			
			To estimate the number of copies of $K_{2,t}$ formed in this way, let $\deg(e)$ denote the number of $K_r$'s containing the edge $e$.  By \Cref{lem:commonNeighbors}, the two vertices of $e$ have at least $\deg(e)^{1/(r-2)}$ common neighbors.  Thus
			\[\N(K_{2,t},G)\ge \half \sum_{e\in E(G)} {\deg(e)^{1/(r-2)}\choose t}\ge \half\sum_{e\in E(G)}\left(\frac{\deg(e)^{t/(r-2)}}{t^t}-1\right),\]
			where this last step used the inequality ${x\choose t}\ge x^t/t^t-1$ valid for all $x$.  
			
			Since $t>r-2$ by hypothesis, the function $x^{t/(r-2)}/t^t-1$ is convex, and hence the expression above is minimized when each $\deg(e)$ is equal to the average value
			\[\ell:={r\choose 2} k n^{3/2} e(G)^{-1},\]
			so we find
			\[\N(K_{2,t},G)\ge \half e(G)\left(\frac{\ell^{t/(r-2)}}{t^t}-1\right).\]
			Note that if $e(G)\le \max(C,k^{1/2})n^{3/2}$ then $\ell\ge 2t^t$ for $k$ sufficiently large, meaning $\frac{\ell^{t/(r-2)}}{t^t}-1=\Omega\left(\ell^{t/(r-2)}\right)$.  In total then, we find
			\[\N(K_{2,t},G)=\Om(e(G) \ell^{t/(r-2)})=\Omega\left(k^{\f{t}{r-2}}n^{\f{3t}{2(r-2)}}e(G)^{\f{r-2-t}{r-2}}\right)\]
			as desired.
		\end{proof}
		We now split up our analysis based off of the value of $k$.  Recalling the value of $C$ defined before \eqref{eq:graphDense}, we first consider the case that $k\le C^{\f{(2t-1)(r-2)+t}{t}} n^{\f{r-2}{2t}}$.  If  $e(G)\ge Cn^{3/2}$, then by \eqref{eq:graphDense} we have \[\N(K_{2,t},G)=\Om(n^2)=\Om\left(k^{\f{t}{r-2}}n^{3/2}\right),\] with the last step using our assumption on $k$, proving the result.  If instead $e(G)\le C n^{3/2}$, then \eqref{eq:expandTriangles} gives a lower bound of $\Om(k^{\f{t}{r-2}}n^{\f{3}{2}})$ as desired.
		
		From now on we assume $k\ge  C^{\f{(2t-1)(r-2)+t}{t}} n^{\frac{r-2}{2t}}$.  First  consider the case \[e(G)\ge k^{\f{t}{(2t-1)(r-2)+t}}n^{\f{(6t-4)(r-2)+3t}{(4t-2)(r-2)+2t}}\ge C n^{3/2},\] where this last inequality holds by our assumption on $k$.  By \eqref{eq:graphDense} we find
		\[\N(K_{2,t},G)=\Om\left( k^{\f{2t^2}{(2t-1)(r-2)+t}}n^{\f{(6t^2-4t)(r-2)+3t^2}{(2t-1)(r-2)+t}}\cdot n^{2-3t}\right)=\Om\left(k^{\f{2t^2}{(2t-1)(r-2)+t}}n^{\f{(3t-2)(r-2)+2t}{(2t-1)(r-2)+t}}\right),\]
		giving the desired bound.  If instead $e(G)\le k^{\f{t}{(2t-1)(r-2)+t}}n^{\f{(6t-4)(r-2)+3t}{(4t-2)(r-2)+2t}}\le k^{1/2}n^{3/2}$, then by \eqref{eq:expandTriangles} we have
		\begin{align*}
			\N(K_{2,t},G) & =\Om\left(\left[k^{t}n^{\f{3t}{2}}\cdot k^{\f{(r-2-t)t}{(2t-1)(r-2)+t}} n^{\f{(r-2-t)((6t-4)(r-2)+3t)}{(4t-2)(r-2)+2t}}\right]^{1/(r-2)}\right) \\ & =\Om\left(\left[k^{\f{2t^2(r-2)}{(2t-1)(r-2)+t}}n^{\f{((6t-4)(r-2)+4t)(r-2)}{(4t-2)(r-2)+2t}}\right]^{1/(r-2)}\right) \\ & =\Om\left(k^{\f{2t^2}{(2t-1)(r-2)+t}}n^{\f{(3t-2)(r-2)+2t}{(2t-1)(r-2)+t}}\right),
		\end{align*}
		again giving the desired bound.
	\end{proof}

	\section{Constructions}
	
	In this section, we construct graphs $G$ with many $K_r$'s but few copies of $K_{2,t}$.  Motivated by \Cref{prop:trees} and \Cref{lem:Kruskal} whose extremal constructions were unions of cliques, it is perhaps reasonable to consider $G$ which are union of cliques.  And indeed, our constructions will consist of two different $G$ of this form: one coming from the union of random cliques, the other from the union of cliques which avoids certain structures.
	
	\subsection{Uniform Random Cliques}\label{sub:uniform}
	In this subsection we prove \Cref{thm:randomCliques} by constructing graphs $G$ which contain many copies of $K_r$ and few copies of $F$ when $F$ is 2-balanced.  Intuitively, the graph $G$ will be formed by taking the union of roughly $u$ cliques of size $m$ chosen uniformly at random.  
	
	More precisely, given a real number $u$ and integers $m,n$, define the random clique graph $G_{u,m,n}$ as follows.  Let $C_1,\ldots,C_{n\choose m}$ be an enumeration of the $m$-element subsets of $[n]$ and let $B_1,\ldots,B_{n\choose m}$ be i.i.d.\ Bernoulli random variables with probability of success $ u {n\choose m}^{-1}$.  Define $G_{u,m,n}$ to be the graph on $[n]$ with edge set 
	\begin{equation}\label{union}
		\bigcup_{i:B_i=1} {C_i\choose 2}.
	\end{equation}
	
	We will prove two properties about the random clique graph $G_{u,m,n}$: that it contains a relatively large number of $K_r$'s, and that it contains few copies of other $F$ (provided $u$ and $m$ are chosen appropriately).  We begin with the clique estimate.

	%Before getting into the details of the proof, we record a binomial tail bound that will be needed in the proof.
	%	\begin{lem}\cite[Theorem 2.1]{JLR}\label{binup}
		%		If $X \sim \text{Bin}(n,p)$,
		%		\[\Pr(X \geq np + t) \leq \exp[-t^2/(2np + 2t/3)].\]
		%	\end{lem}
	
	\begin{lem}\label{lem:Krcount}
		For all $r\ge 2$, there exists $\delta=\delta(r)>0$ such that if $u\ge 1,m\ge 2r$ and $u m^r\le n^r$, then \[\Pr(\N(K_r,G_{u,m,n})>\delta u m^r)>\delta.\]
	\end{lem}
	\begin{proof}
		We record the following binomial tail bound that will be needed in the proof, see for example \cite[Theorem 2.1]{JLR}: if  $X \sim \text{Bin}(n,p)$, then
		\begin{equation}\Pr(X \geq np + t) \leq \exp[-t^2/(2np + 2t/3)].\label{binup}\end{equation}
		
		Returning to the main proof, let $X = \binom{m}{r}|\{i : B_i = 1\}|$ (which counts the number of $r$-cliques in $G$ \emph{with multiplicity}), so $X \sim \binom{m}{r}\text{Bin}({n\choose m},u {n\choose m}^{-1})$.  Using the general fact that $\Pr(\text{Bin}(M,p) \geq Mp/2) \geq 1/2$ if $Mp \geq 1$ (which follows from e.g.\ \cite{lord_2010}), together with $u\ge 1$ and $m\ge 2r$ gives
		\begin{equation}\label{Xlower}
			\Pr\left(X\ge \frac{um^r}{2^{r+1}r!} \right)\ge \Pr\left(X \geq \half u{m\choose r}\right) \geq \half .
		\end{equation}
		
		Similarly, if $\lambda(A) = |\{i : B_i = 1, A \subset C_i \}|$ counts the multiplicity of a fixed $r$-clique $A$, then $\lambda \sim \text{Bin}({n-r\choose m-r}, u{n\choose m}^{-1})$.  This random variable has expectation 
		\[u (m)_r/(n)_r \lesssim u m^r/n^r  =: \mu .\]
		and this is at most $1$ by hypothesis. Note that \eqref{binup} says that for $t\ge 6$, 
		\[\Pr(\lambda(A) \geq 1+t)\leq \exp[-t^2/(2+ 2t/3)] \leq \exp[-t].\] 
		
		Thus, if $Y_i := |\{A : 2^i \leq \lambda(A) < 2^{i+1}\}|$, we find that $\E[Y_i]\le {n\choose r} \mu \exp[-2^{i-1}]$ for, say $i\geq 10$. By Markov's inequality, we have 
		\begin{equation}\label{Ybound}
			\Pr\left(Y_i \geq \binom{n}{r}\mu \exp[-2^{i-2}] \right) < \exp[-2^{i-2}]\le \exp[-i]
		\end{equation}
		for large enough $i$. Letting $L$ be a large constant to be determined, we have
		\[X =\sum_A \lambda(A) < \sum_i 2^{i+1}Y_i \leq \sum_{i=1}^L 2^{i+1}Y_i + \sum_{i > L} 2^{i+1}Y_i.\]
		
		By \eqref{Ybound}, the last sum is at most $2\binom{n}{r}\mu\sum_{i > L+1} (2/e)^i \leq 4 (2e^{-1})^L um^r/r!$ with probability at least $1 - \sum_{i>L}\exp[-i]$. Thus, using \eqref{Xlower} and taking $L$ sufficiently large, we have with probability at least $1/3$ that 
		\[\f{um^r}{2^{r+2}r!} \leq \sum_{i=1}^L 2^{i+1}Y_i \leq 2^{L+1} \sum_{i} Y_i = 2^{L+1}\N(K_r, G),\]
		which completes the proof.    
	\end{proof}
	
	We next aim to prove the random clique construction contains few copies of $F$ for certain ranges of parameters.
	
	\begin{lem}\label{lem:exp}
		Let $F$ be a 2-balanced graph with $e(F)\ge 2$.  If $u,m,n$ with $m\le n$ are such that $um^2<n^2$ and $u(m/n)^{2-\frac{v(F)-2}{e(F)-1}}> 1$, then
		\[\E[\N(F,G_{u,m,n})]=O\left((u m^2 n^{-2})^{e(F)}n^{v(F)}\right ).\]
	\end{lem}
	The condition $um^2<n^2$ intuitively means the cliques of $G_{u,m,n}$ will be close to edge disjoint.  The condition $u(m/n)^{2-\frac{v(F)-2}{e(F)-1}}> 1$ is best possible for the conclusion to hold, as otherwise the count $u {m\choose v(F)}$ coming from copies of $F$ within a single clique will be larger.
	
	We need a few technical definitions to prove \Cref{lem:exp}.  These definitions are based off the observation that for a given copy $\tilde{F}$ of $F$ to be present in $G_{u,m,n}$, there must exist some (minimal) set of $m$-subsets $\{C_{i_1},\ldots,C_{i_s}\}$ which cover the edges of $\tilde{F}$ and which are all present as cliques in $G_{u,m,n}$.
	
	With this in mind, given a graph $\tilde{F}\sub K_n$, we say that a family $\c{C}$ of $m$-subsets of $K_n$ is an \textit{$\tilde{F}$-covering} if $E(\tilde{F})\sub \bigcup_{C\in \c{C}} {C\choose 2}$, if each $C\in \c{C}$ contains at least one edge of $\tilde{F}$, and if $C\cap V(\tilde{F})\ne  C'\cap V(\tilde{F})$ for all distinct $C,C'\in \c{C}$.  Given $G_{u,n,n}$ and a family of $m$-subsets $\c{C}=\{C_{i_1},\ldots,C_{i_s}\}$ of $K_n$, we let $\c{B}(\c{C})$ denote the event that $B_{i_j}=1$ for all $1\le j\le s$ (that is, this is the event that each of the $C_{i_j}$ appear in the union of \eqref{union}).   Given a graph $F$, we let $Z(F,G_{u,m,n})$ denote the set of pairs $(\tilde{F},\c{C})$ such that $\tilde{F}\sub K_n$ and $\c{C}$ is an $\tilde{F}$-covering with $\c{B}(\c{C})$ occurring.
	
	The crucial observation is the following.
	
	\begin{lem}\label{lem:Z}
		For all graphs $F$, we have $\N(F,G_{u,m,n})\le Z(F,G_{u,m,n})$.
	\end{lem}
	\begin{proof}
		Observe that if $\tilde{F} \subseteq G_{u,m,n}$ is isomorphic to $F$, then there exists an $\tilde{F}$-covering $\c{C}$ such that $\BB(\c{C})$ occurs; namely by taking a minimal set of $m$-subsets $C_{i_j}$ with $B_{i_j}=1$ that contain the set of edges of $\tilde{F}$ 
		(which must exist if $\tilde{F}\sub G_{u,m,n}$).  Thus for each subgraph $\tilde{F}\sub G_{u,m,n}$ counted by $\N(F,G_{u,m,n})$, there exists at least one pair $(\tilde{F},\c{C})$ counted by $Z(F,G_{u,m,n})$, proving the bound.
	\end{proof}
	With \Cref{lem:Z} in mind, we see that \Cref{lem:exp} will immediately be implied by the following result.
	
	\begin{lem}\label{lem:expTech}
		Let $F$ be a 2-balanced graph with $e(F)\ge 2$.  If $u,m,n$ with $m\le n$ are such that $um^2<n^2$ and $u(m/n)^{2-\frac{v(F)-2}{e(F)-1}}> 1$, then
		\[\E[Z(F,G_{u,m,n})]=O\left((u m^2 n^{-2})^{e(F)}n^{v(F)}\right ).\]
	\end{lem}
	\begin{proof}
		Fix any $\tilde{F}\sub K_n$ isomorphic to $F$.  Since there are at most $n^{v(F)}$ choices for $\tilde{F}$, we see that it suffices to prove that the expected number of $\tilde{F}$-coverings $\c{C}$ for which $\c{B}(\c{C})$ occurs is at most $O((u m^2 n^{-2})^{e(F)})$.  For this we need a few more definitions.
		
		Given a family $\c{A}$ of subsets of $V(\tilde{F})$, we say that an $\tilde{F}$-covering $\c{C}$ has \textit{type $\c{A}$} if \[\{C\cap V(\tilde{F}):C\in \mathcal{C}\}=\c{A}.\]  We say that a family $\c{A}$ is \textit{valid} if $\tilde{F}\sub \bigcup_{A\in \c{A}} {A\choose 2}$ and if  each $A\in \c{A}$ contains at least one edge of $\tilde{F}$.  Observe that by definition, if $\c{C}$ is an $\tilde{F}$-covering, then $\c{C}$ is of type $\c{A}$ for some valid $\c{A}$ with $|\c{A}|=|\c{C}|$.  Given a valid $\c{A}$, we define
		\[w(\c{A}):=u^{|\c{A}|} (m/n)^{\sum_{A\in \c{A}} |A|}.\]
		
		\begin{claim}
			For any family $\c{A}$, let $\c{T}(\c{A})$ denote the number of $\tilde{F}$-coverings $\c{C}$ of type $\c{A}$ such that $\c{B}(\c{C})$ occurs.  Then 
			\begin{equation}\label{sep}
				\E[\c{T}(\c{A})] \leq w(\c{A}).
			\end{equation}
		\end{claim}
		\begin{proof}
			Let $\c{A}=\{A_1,\ldots,A_s\}$, and let $\c{S}$ be the set of $\tilde{F}$-coverings $\c{C}$ of type $\c{A}$. Since $|\c{S}|$ is at most the number of tuples $(C_{i_1},\ldots,C_{i_s})$ such that each $C_{i_j}$ is an $m$-subset containing $A_j$, we find that
			\[|\c{S}|\le \prod_{j=1}^s {n-|A_j|\choose m-|A_j|}= \prod_{j=1}^s  {n\choose m}\binom{m}{|A_j|}/\binom{n}{|A_j|} \le  \prod_{j=1}^s {n\choose m} (m/n)^{|A_j|},\]
			Now, any fixed set $\c{C}=\{C_{i_1},\ldots, C_{i_s}\}$ of distinct $m$-subsets has $\c{B}(\c{C})$ occurring with probability $u^s {n\choose m}^{-s}$, so by a union bound we see that
			\[\E[\c{T}(\c{A})] \leq u^s {n\choose m}^{-s}  |\mathcal{S}|\leq  \prod_{j=1}^s u(m/n)^{|A_j|}=w(\c{A}),
			\]
			proving the claim.
		\end{proof}
		With this claim, we see that the expected number of $\tilde{F}$-coverings for which $\c{B}(\c{C})$ occurs is at most
		\[\sum_{\c{A}} \E[\c{T}(\c{A})]\le 2^{2^{v(F)}}\cdot \max_{\c{A}} w(\c{A}),\]
		where the sum and the maximum range over all valid families $\c{A}$. Thus to prove the result, it suffices to show 
		\begin{equation}\max_{\c{A}} w(\c{A})\le (u m^2 n^{-2})^{e(F)},\label{maxbound}\end{equation}
		where the maximum ranges over all valid families $\c{A}$.  We will call any $\c{A}$ achieving the maximum in \eqref{maxbound} a \textit{maximizer}, and our goal will be to show that the only maximizers are those with $A\in E(F)$ for all $A\in \c{A}$.  We do this through the following two claims.
		\begin{claim}
			Every maximizer $\c{A}$ has $|A \cap B|\le 1$ for all distinct $A, B\in \c{A}$.
		\end{claim}
		\begin{proof}
			Let $\c{A}$ be a maximizer and assume for contradiction that $|A\cap B|\ge 2$ for some distinct $A,B\in \c{A}$.  In this case, the set $\c{A}'=(\c{A}\sm\{A,B\})\cup \{A\cup B\}$ is a valid family which satisfies $|\c{A}'|=|\c{A}|-1$ and \[\sum_{D\in \c{A}'} |D|=-|A\cap B|+\sum_{D\in \c{A}} |D|,\] since $|A\cup B|=|A|+|B|-|A\cap B|$.  Thus
			\[w(\c{A}')=u^{-1} (m/n)^{-|A\cap B|}w(\c{A})>w(\c{A}),\]
			where the inequality used $|A\cap B|\ge 2$ together with $um^2< n^2$ and $m\le n$ applied $|A\cap B|-2$ times.  This contradicts $\c{A}$ being a maximizer, giving the result.
		\end{proof}
		\begin{claim}
			Every maximizer $\c{A}$ has $A\in E(F)$ for all $A\in \c{A}$.
		\end{claim}
		\begin{proof}
			Assume for contradiction that $\c{A}$ is a maximizer with $|A|>2$ for some $A\in \c{A}$.  Let $\c{E}$ denote the set of edges of $F':=F[A]$, noting that $\sum_{e\in \c{E}} |e|=2e(F')$ and $|A|=v(F')$. Let $\c{A}'=(\c{A}\setminus \{A\})\cup \c{E}$. Observe that $\c{A}'$ is also a valid family with $|\c{A}'|=|\c{A}|+e(F')-1$ (noting that we have $\c{E}\cap \c{A}=\emptyset$ by the previous claim since $A\in \c{A}$ and $\c{A}$ is a maximizer).  Thus, 
			\[w(\c{A}')=u^{e(F')-1}(m/n)^{2e(F')-v(F')} w(\c{A})>w(\c{A}),\]
			where the last step used 
			\[u (m/n)^{2-\frac{v(F')-2}{e(F')-1}}\ge u(m/n)^{2-\frac{v(F)-2}{e(F)-1}}> 1,\]
			with the first inequality using that $F$ is 2-balanced (and $m\leq n$) and the last inequality using the hypothesis of the lemma. This contradicts $\c{A}$ being a maximizer, so we must have $|A|\le 2$ for all $A\in \c{A}$.  Moreover, each $A\in \c{A}$ must contain an edge of $F$ by definition of $\c{A}$ being valid, giving the result.
		\end{proof}
		The claims above imply $\c{A}=E(F)$ is the only maximizer, in which case $w(\c{A})=u^{e(F)}(m/n)^{2e(F)}$, giving \eqref{maxbound} and hence the result.
	\end{proof}
	
	In addition to proving \Cref{lem:exp}, \Cref{lem:expTech} can be used to derive our general lower bound on $\ex(n,K_r,F)$.
	
	\begin{proof}[Proof of \Cref{thm:TuranBound}]
		Recall that we aim to prove that if $F$ is a 2-balanced graph with $e(F)\ge 2$ and $2\le r<v(F)$ an integer, then
		\[\ex(n, F, K_r) =\Omega(n^{2 - \frac{v(F)-2}{e(F)-1}}).\]
		
		Let $m=r$ and $u = 2(n/m)^{2 - \f{v(F) - 2}{e(F) - 1}} = \Omega(n^{2 - \f{v(F) - 2}{e(F) - 1}})$.  Given a constant $\alpha>0$, we let $G_\al=G_{\al u,m,n}$. Note that the conditions of Lemma~\ref{lem:expTech} apply to $G_1$ (though not necessarily for $G_{\al}$). We aim to show that for small $\al$, we can alter $G_\al$ to make it $F$-free by removing a small portion of its $K_r$'s.
		
		Let $Z_\al=Z(F,G_\al)$.  Since $v(F)>r=m$, any $\tilde{F}$-covering $\c{C}=\{C_{i_1},\ldots,C_{i_s}\}$ of some $\tilde{F}\cong F$ must have $s\ge 2$, so by definition of $Z$  we find $\E[Z_\al]\le \al^2 \E[Z_1]$.  By Lemma~\ref{lem:expTech} we see
		\[\E[Z_1] = O((um^{2}n^{-2})^{e(F)}n^{v(F)}) = O(n^{2 - \f{v(F) - 2}{e(F) - 1}}) = O(u),\]
		and thus
		\[\E[Z_\alpha] = O(\alpha^2 u).\]
		
		Recall that $B_1,B_2,\ldots$ are the Bernoulli random variables associated to $G_\al$ such that the $i$th clique $C_i$ is included in the union \eqref{union} for $G_\al$ if $B_i=1$.  Let $Y(G_\al)$ denote the number of $i$ such that $B_i=1$, noting that $\N(K_r,G_\al)\ge Y(G_\al)$ and that $\E[Y(G_\al)]=\al u$.  Thus we find $\E[Y(G_\al)] - \E[Z_\alpha] > \alpha u - O(\alpha^2 u)$, so for all $\al$ there is a realization $G'_\al$ of $G_\al$ with \[Y(G_\al') - Z(G'_\al) > \alpha u - O(\alpha^2 u).\] Taking $\alpha$ sufficiently small and removing one clique from each pair $(\tilde{F}, \{C_{i_1},\ldots \})$ in $G'_\al$ counted by $Z(F,G'_\al)$ (meaning we remove the clique from the union \eqref{union}, which does not necessarily remove any edges from $G'_\al$) results in an $F$-free graph $G''_\al$ with \[\N(K_r, G''_\al)\ge Y(G''_\al) = \Omega(u) = \Omega(n^{2 - \f{v(F) - 2}{e(F) - 1}}).\]

		\iffalse 
		Since $\E[\N(K_r, G_\al)] = \alpha u$, \BG{I'm not sure we want to use $\N(K_r, G_\al)$, as this can decrease by more than one by removing a clique from the union (if there happen to be random other copies of $K_r$ that appear). We actually want something like `the number of cliques $C_i$ that were initially picked by the random process'.} \SS{Great point; this isn't an issue right here because we can just put a lower bound, but this an error in the next part in the argument since when we remove a clique we might delete more things.  Probably yes we should define this thing as Ben notes, then argue that we can find some $G''$ where this is large which implies the clique count is large} we have $\E[\N(K_r, G_\al)] - \E[Z_\alpha] > \alpha u - O(\alpha^2 u)$, so for all $\al$ there is a realization $G'_\al$ of $G_\al$ with \[\N(K_r, G'_\al) - Z(G'_\al) > \alpha u - O(\alpha^2 u).\] Taking $\alpha$ sufficiently small and removing one clique from each pair $(\tilde{F}, \{C_{i_1},\ldots \})$ in $G'_\al$ counted by $Z(F,G'_\al)$ (meaning we remove the clique from the union \eqref{union}, which does not necessarily remove any edges from $G'_\al$) results in an $F$-free graph $G''_\al$ with \[\N(K_r, G''_\al) = \Omega(u) = \Omega(n^{2 - \f{v(F) - 2}{e(F) - 1}}).\] \fi
	\end{proof}

	With this all established, we can now prove our main result for this subsection.

	\begin{proof}[Proof of \Cref{thm:randomCliques}]
		Recall that we wish to prove that if $F$ is a 2-balanced graph with $e(F)\ge 2$ and $2\le r<v(F)$ is an integer, then for all $1\le N\le {n\choose r}$, there exists an $n$-vertex graph $G$ with $\N(K_r,G)=\Om(N)$ and with
		\[\N(F,G)=O\left((N n^{-r})^{\frac{e(F)(v(F)-2)}{(r-2)(e(F)-1)+v(F)-2}} n^{v(F)}\right).\]
		
		We first consider some trivial cases.  If $F$ is the disjoint union of $K_2$'s, then one can check that the bound above is achieved by taking $G$ to be a clique on $N^{1/r}$ vertices.  If $F$ has an isolated vertex $x$, then $F':=F-x$ has at least 3 vertices (since $e(F)\ge 2$) and $\frac{e(F')-1}{v(F')-2}>\frac{e(F)-1}{v(F)-2}$, contradicting $F$ being 2-balanced.  Thus we can assume $F$ has no isolated vertices and at least one component which is not a $K_2$, from which it follows that \[2e(F)>v(F)\ge 3,\]
		where here we used that no isolated vertices implies $2e(C)\ge v(C)$ for all components $C$, and the component which is not a $K_2$ gives a strict inequality.
		
		The result is also trivial if $N\le \ex(n,K_r,F)$, as in this case there exist $F$-free graphs with the desired number of copies.  Thus by \Cref{thm:TuranBound} we can assume \[N\ge c n^{2-\frac{v(F)-2}{e(F)-1}}\]
		for some $c\le 1$.  Similarly the result is trivial if  $N=\Om(n^r)$ by taking $G=K_n$, so we can assume $N$ is at most a small constant times $n^r$ (with this constant depending on $F,r,c$).
		
		With all these assumptions above in mind, we define $C=2r c^{-\frac{e(F)-1}{(r-2)(e(F)-1)+v(F)-2}}$ and take \[u=2(N n^{-r})^{\frac{v(F)-2e(F)}{(r-2)(e(F)-1)+v(F)-2}},\hspace{3em} m=C(N n^{-r})^{\frac{e(F)-1}{(r-2)(e(F)-1)+v(F)-2}}\cdot n.\]   
		The lower bound $N\ge c n^{2-\frac{v(F)-2}{e(F)-1}}$ immediately gives $m\ge 2r$.  Observe that $v(F)-2e(F)<0$ and $(r-2)(e(F)-1)+v(F)-2>0$ due to the bound $2e(F)>v(F)\ge 3$ above and $r\ge 2$, which in particular implies $C\ge 1$ since $c\le 1$.  With this and our assumption that $N$ is at most a small constant times $n^r$, we observe that $m<n$, that $u\ge 2$, 
		\[1\le um^r=2CN\le n^r,\] 
		\[u(m/n)^{2-\frac{v(F)-2}{e(F)-1}}=2 C^{2-\frac{v(F)-2}{e(F)-1}}>1,\]
		and
		\[um^2 n^{-2}=2C^2(N n^{-r})^{\frac{v(F)-2}{(r-2)(e(F)-1)+v(F)-2}}<1.\]
		Now consider $G=G_{u,m,n}$ and let $\del$ be the constant from \Cref{lem:Krcount}.  By applying Markov's inequality to \Cref{lem:exp}, we find that \[\N(F, G) = O\left((N n^{-r})^{\frac{e(F)(v(F)-2)}{(r-2)(e(F)-1)+v(F)-2}} n^{v(F)}\right)\] with probability at least $1 - \del/2$.  By \Cref{lem:Krcount}, we see $\N(K_r,G)=\Om(N)$ with probability at least $\del$.  In particular, $G$ satisfies the desired properties with positive probability, showing  such a graph exists.
	\end{proof}

	\subsection{Cliques from Bipartite Graphs}\label{sec:bipartiteConst}
	Throughout this subsection we work with bipartite graphs $B$ with ordered bipartitions $(U,V)$.
	
	The intuition for our construction is as follows. We again consider a graph $G$ formed by taking the union of roughly $u$ cliques of size $m$ for some parameters $u,m$.  We can not have $m$ larger than what it was in the proof for \Cref{thm:randomCliques}, as otherwise the number of copies of $K_{2,t}$ contained within the $m$ cliques will be too large.  Thus we must take $m$ to be smaller and $u$ to be larger.  If we put the $u$ cliques down uniformly at random, then $G$ would contain too many copies of $K_{2,t}$ which have each edge contained in a distinct clique (intuitively because $G$ behaves locally like $G_{n,p}$).  Ideally then, we want to place our cliques down so that there exists no $K_{2,t}$ with each edge contained in a distinct clique.  For this the following will be useful.
	
	\begin{defn}
		Given a bipartite graph $B$ with  ordered bipartition $(U,V)$, we define the \textit{clique graph} $K(B)$ to be the graph with vertex set $V$ such that $v,v'\in V$ are adjacent if and only if $v,v'$ have a common neighbor in $U$.  Equivalently, $K(B)$ is formed by taking the union\footnote{A helpful mnemonic is that $U$ is the set of cliques that we Union over.} of the cliques $N_B(u)$ with $u\in U$. 
	\end{defn}
	Unwinding the intuition from above; we want to find a bipartite $B$ which avoid subdivisions of $K_{2,t}$ (as these correspond to edges belonging to distinct cliques in $K(B)$), or equivalently, to avoid having $t$ paths of length 4 between any two vertices of $V$ in $B$.  And indeed, the following shows that this is essentially all we need.
	
	\begin{lem}\label{lem:K2lcount}
		If $B$ is a bipartite graph with bipartition $(U,V)$ such that every vertex of $U$ has degree at most $d$ and such that there are at most $\ell\ge 2$ distinct paths of length at most 4 between any two vertices of $V$, then for all $t>\ell$ we have
		\[\N(K_{2,t},K(B))\le \ell^{2t} d^{2+t}|U|.\]
	\end{lem}
	
	Note that the number of $K_{2,t}$'s  within any of the $N_B(u)$ cliques of $K(B)$ is at most $d^{2+t}|U|$, so the lemma says that this trivial count is essentially correct.
	\begin{proof}
		We first claim that if $v_1,v_2,w_1,\ldots,w_t$ form a $K_{2,t}$ in $K(B)$ with $v_1,v_2$ adjacent to all of the $w_j$ vertices, then $v_1v_2\in E(K(B))$ (i.e.\ they have a common neighbor in $B$).  Indeed, for each edge $v_iw_j\in E(K(B))$ there exists a $u_{i,j}\in U$ which is adjacent to both of these vertices.  If $u_{1,j}=u_{2,j}$ for any $j$ then we are done, so we assume this is not the case.  Thus $v_1u_{1,j}w_ju_{2,j}v_2$ is a path of length 4 from $v_1$ to $v_2$ in $B$, and each of these $t>\ell$ paths are distinct since the $w_j$ vertices are distinct.  This is impossible by our condition on $B$, so the claim follows.
		
		For any pair of vertices $v_1,v_2\in K(B)$, we claim that there are at most $\ell(d+1)$ vertices $w$ which are adjacent to both $v_1,v_2$ in $K(B)$.  Indeed, for any such vertex $w$ there must exist two (possibly non-distinct) vertices $u^w_1,u_2^w\in U$ such that $w,v_i\in N_B(u_i^w)$.  If $u^w_1=u^w_2:=u$, then in particular $u$ is a common neighbor of $v_1,v_2$ in $B$.  By assumption there are at most $\ell $ such vertices $u$, and each of them have at most $d$ neighbors $w$.  Thus the number of common neighbors with $u^w_1=u_2^w$ is at most $\ell d$.  On the other hand, for each distinct $w$ with $u^w_1\ne u^w_2$, there exists a distinct path $v_1u_1^wwu_2^wv_2$ of length 4 from $v_1$ to $v_2$ in $B$.  By assumption, at most $\ell$ such vertices can exist, giving the claim.
		
		By these two claims, every $K_{2,t}$ in $K(B)$ can be formed by first choosing $v_1,v_2$ adjacent in $K(B)$ (i.e.\ with a common neighbor in $U$) and then choosing some $t$ of the at most $\ell(d+1)\le 2\ell d$ common neighbors they have in $K(B)$.  In total the number of ways of doing this is at most
		\[d^2|U|\cdot {2\ell d\choose t}\le \ell^{2t} d^{2+t}|U|.\]
	\end{proof}
	It remains to find graphs $B$ avoiding the structures in \Cref{lem:K2lcount} which are  ``dense'' (so that $K(B)$ will have many $K_r$'s).  To this end, we define $\ex(m,n,\c{P}_{\le 4}^{\ell+1})$ to be the maximum number of edges that a bipartite graph $B$ with bipartition $(U,V)$ and $|U|=m,|V|=n$ can have if there are at most $\ell$ distinct paths of length at most 4 between any two vertices of $V$, and in this case we say $B$ \textit{avoids} $\c{P}_{\le 4}^{\ell+1}$. It is relatively easy to upper bound this extremal number by adapting an argumenent of Bukh and Conlon~\cite[Lemma 1.1] {bukh2018rational}.
	\begin{lem}\label{lem:P4pExt}
		\[\ex(m,n,\mathcal{P}_{\le 4}^{\ell+1}) < 4 \ell^{1/4} n^{3/4} m^{1/2} + 10m + 10n.\]
	\end{lem}
	\begin{proof}
		Let $B = (U,V, E)$ be a bipartite graph with $|U|=m$, $|V|=n$, and $|E|=e$ which avoids $\c{P}_{\le 4}^{\ell+1}$.  If $e<10m$ or $e<10n$ then the result follows, so assume this is not the case.
		
		Let $U'\subseteq U$ be the set of vertices with minimum degree at least $e/(2m)$.  For $v\in V$, let $d'_v=|N(v)\cap U'|$, and note $\sum_{v\in V} d_v' \geq e/2$, as at most $e/2$ edges can be incident to vertices in $U-U'$. Let $X$ denote the number of labelled copies of $P_4$ in $U' \cup V$ with endpoints in $V$. One can lower bound $X$ by greedily embedding copies of $P_4$ starting with the middle vertex to get
		\[X \geq \sum_{v\in V} \binom{d_v'}{2}(e/(2m) - 1)(e/(2m) - 2) \geq n\binom{e/(2n)}{2} \left(\frac{e}{4m}\right)^2 \geq \frac{e^4}{2^8 m^2n},\]
		where the second and third inequalities use that $e \geq 10m$ and $e \geq 10n$ respectively.
		On the other hand, there are $\binom{n}{2}$ choices for the endpoints of each path, and each pair may appear at most $\ell$ times as endpoints, so 
		\[X \leq \ell n^2/2.\]
		Comparing the lower and upper bounds gives
		\[e^4 \leq 2^7 \ell n^3 m^2,\]
		which implies the lemma.
	\end{proof}
	
	The next result shows that whenever $\ex(m,n,\mathcal{P}_{\le 4}^{\ell+1})$ is close to the upper bound of \Cref{lem:P4pExt} we can find a graph with many $K_r$'s and few $K_{2,t}$'s.
	
	\begin{lem}\label{lem:BipToSatgraph}
		Fix $r\ge 3$, $\ell\ge 2$ and $t> \ell,r-2$. Assume there exists a $c>0$ such that for all $\ep$ with $0\le \ep\le \frac{r-2}{2t}$ we have $\ex(n^{3/2-2\ep},n,\c{P}_{\le 4}^{\ell+1})>2cn^{3/2-\ep}$.  Then for all $0\le \ep\le \frac{r-2}{2t}$ there exists an $n$-vertex graph $G$ with $\Om_c(n^{3/2+(r-2)\ep})$ $r$-cliques and $O_c(n^{3/2+t\ep })$ $K_{2,t}$'s.
	\end{lem}
	\begin{proof}
		We may assume $n$ is sufficiently large in terms of $c$, as otherwise we can take $G=K_n$ to give the result.  Let $B = (U,V, E)$ be a bipartite graph with $|U| =m:=n^{3/2-2\ep}$ and $|V| = n$ showing $\ex(n^{3/2-2\ep},n,\c{P}_{\le 4}^{\ell+1})>2cn^{3/2-\ep}$. Let $W\subseteq U$ be the set of vertices $w$ with $\deg_B(w) > D\ell n^{\ep}$ for some large $D$ to be determined.
		
		\begin{claim}
			
			$e(B[U\sm W, V]) > e(B)/2$.\end{claim}
		\begin{proof}
			Consider the bipartite graph $B_W$ induced by $W \cup V$. We are done if $e(B_W)\le cn^{3/2 - \ep}$. Suppose this is not the case, and set $s = |W|$.  Using $n$ sufficiently large in terms of $c$ and that $\ep<1/2$, we find \[e(B_W) > cn^{3/2 - \ep} > 100s + 100n,\] so Lemma~\ref{lem:P4pExt} gives \[s^{1/2} > \frac{e(B_W)}{5\ell^{1/4} n^{3/4}}> \frac{cn^{3/2 - \ep}}{5 n^{3/4} \ell^{1/4}},\] i.e.\ $s > c^2n^{3/2 - 2\ep}/(25 \ell^{1/2})$. Therefore, \[e(B_W) > s \cdot D \ell n^{\ep} > \frac{n^{3/2 - 2\ep}}{25 \ell^{1/2}}\cdot  D \ell n^{\ep} = \ell^{1/2} n^{3/2 - \ep} (D/25),\] which is more than $4 \ell^{1/4} n^{3/4}m^{1/2} + 10(m+ n)$ for large enough $D$, contradicting Lemma~\ref{lem:P4pExt}.
		\end{proof}
		Let $U' = U \sm W$, $B' = B[U'\cup V]$, and $|U'| = m'$. By the claim, $e(B') > e(B)/2$. Let $G = K(B')$ (recalling the definition of $K(B')$ above Lemma~\ref{lem:K2lcount}). We claim $G$ satisfies the conclusion of the lemma.
		
		Because $B'$ avoids $\c{P}_{\le 4}^{\ell+1}$, every $r$-set in $V$ is contained in at most $\ell$ neighborhoods, so the number of $K_r$'s in $G$ is at least
		\[\sum_{u\in U'} \binom{\deg_{B'}(u)}{r}\ell^{-1} \geq m'\binom{e(B)/(2m')}{r}\ell^{-1}\geq \frac{c^r n^{3/2 +(r-2)\ep}}{2^r r! \ell}\]
		for sufficiently large $n$, with this last step using $m'\le m=n^{3/2-2\ep}$.  Note that Lemma~\ref{lem:K2lcount} with the assumption that $B'$ avoids $\c{P}_{\le 4}^{\ell+1}$ implies  $G$ contains at most $O((2\ell)^{2t} (D\ell n^\ep)^{2 + t}n^{3/2 - 2\ep}) = O(n^{3/2 + t\ep})$ copies of $K_{2,t}$.
	\end{proof}
	
	With this we can prove our upper bound result for $K_{2,t}$.
	
	\begin{proof}[Proof of \Cref{thm:graphK2t} Upper Bound]
		Recall that we wish to prove if $t\ge r-1\ge 2$ and either $k\ge n^{\frac{r-2}{2t}}$ or if there exists a bipartite graph $B$ on $U\cup V$ such that $|U|=k^{-2/(r-2)}n^{3/2},\ |V|=n,\ e(B)=\Om(k^{-1/(r-2)} n^{3/2})$, and such that every pair of vertices in $V$ has less than $t$ paths of length at most 4 between them; then there exists an $n$ vertex graph $G$ with $\N(K_r,G)=\Om(kn^{3/2})$ and with
		\[\N(K_{2,t},G)= O(\min\{k^{\f{t}{r-2}}n^{3/2+o(1)},k^{\f{2t^2}{(2t-1)(r-2)+t}}n^{\f{(3t-2)(r-2)+2t}{(2t-1)(r-2)+t}}\}).\]
		If $k\ge n^{\frac{r-2}{2t}}$ then the second term achieves the minimum, and in this case the general bound of \Cref{thm:randomCliques} gives the result.  
		
		It remains to show that if $k\le n^{\frac{r-2}{2t}}$ and if bipartite graphs $B$ as above exists, then there exist graphs with $\N(K_r,G)=\Om(kn^{3/2})$ and $\N(K_{2,t},G)= O(k^{\f{t}{r-2}}n^{3/2})$.  And indeed, this follows from the previous lemma.
	\end{proof}
	As an aside, we note that the constructions considered here are similar to the constructions used in the recent breakthrough of Mattheus and Verstra\"ete~\cite{mattheus2023asymptotics} for off-diagonal Ramsey numbers.  Indeed, as is made more explicit in \cite{conlon2023ramsey}, the construction of \cite{mattheus2023asymptotics} comes from taking a bipartite graph $B$ which avoids certain subgraphs, forming the clique graph $K(B)$, and then taking a random bipartition of each of its cliques.

	\section{Concluding Remarks}
	In this paper we studied how many copies of a graph $F$ another graph $G$ is guaranteed to have if $G$ contains a given number of $K_r$'s.  While our general result \Cref{thm:randomCliques} gives effective bounds on many $F$, it would be desirable to get a better understanding of the problem for specific choices of $F$.  Below we outline two such directions for problems, and for this we recall the notation $\N(F,G)$ which denotes the number of copies of $F$ in $G$.
	
	\textbf{Tight Bounds for $K_{2,t}$}. 
	\Cref{thm:graphK2t} solved the clique supersaturation problem when $F=K_{2,t}$ and $G$ has $kn^{3/2}$ copies of $K_r$ with $k\ge n^{(r-2)/2t}$.  We believe \Cref{thm:graphK2t} should give tight bounds even when $k<n^{(r-2)/2t}$, and in particular we conjecture the following.
	
	\begin{conj}\label{conj:C4}
		There exists $t_0$ such that for $t\ge t_0$ and $1\le k\le n^{1/2t}$, there exists an $n$-vertex graph $G$ with $\Om(kn^{3/2})$ triangles and with
		\[\N(K_{2,t},G)\le k^{t} n^{3/2+o(1)}.\]
	\end{conj}
	Note that the second half of \Cref{thm:graphK2t} shows such $G$ exists provided there exist bipartite graphs $B$ with parts of sizes $k^{-2} n^{3/2}$ and $n$ which have $e(B)=\Om(k^{-1} n^{3/2})$, and which have fewer than $t$ paths of length 4 between any two vertices in the part of size $n$.  For example, if $t=2$ and $k=n^{1/4}$, then this is equivalent to finding an $n$-vertex bipartite graph with $\Om(n^{5/4})$ edges and which is $C_4$ and $C_8$-free, which is a notoriously open and difficult problem.  More generally, the $t=2$ case 
	%	of $B$ 
	requires finding $C_8$-free unbalanced bipartite graphs of the largest possible density; see for example \cite{naor2005note} for more on this.
	
	The above suggests using \Cref{thm:graphK2t} to solve \Cref{conj:C4} with $t_0=2$ is quite difficult, but there is some hope that one can do this for $t_0$ sufficiently large.  In particular, it is known at $k=n^{1/4}$ that (explicit) bipartite graphs $B$ of this form exist for $t\ge 3$ due to an algebraic construction of Verstra\"ete and Williford~\cite{verstraete2019graphs}, and it is plausible that one could modify their argument to construct $B$ for additional values of $k$.
	
	Another potential avenue is through random polynomial graphs. This approach was used by Conlon~\cite{conlon2019graphs} to show that at $k=n^{1/4}$, there is a $t_0$ such that bipartite graphs $B$ of this form exist for $t\ge t_0$.  By adapting his argument, it is possible to show that for any rational $0\le q\le 1/4$, there exists $t_q$ such that at $k=n^q$ bipartite graphs $B$ of this form exist for $t\ge t_q$.%; see the Appendix in the arXiv version of this article for a somewhat more general version of this result. \SS{Need to decide whether to include the Appendix or not.  I feel like there probably is a way to pitch things so that it makes sense to include, but at the moment I'm inclined to remove it.} \BG{I agree with removing it unless someone has a specific way they want to include it.}
	
	While the $k=n^q$ result above might be of independent interest, it does not suffice for our purposes.  Indeed, the $t_q$ we obtain will typically have $q>1/2t_q$, and hence the range $k=n^q$ falls outside the scope of \Cref{conj:C4}.  Still, it might be possible to prove \Cref{conj:C4} with a more sophisticated approach using random polynomials.

	\textbf{Other Graphs}.  We believe the bounds of \Cref{thm:graphK2t} for $K_{2,t}$ should extend to all theta graphs $F$ with a similar argument.  However, the situation for general $K_{s,t}$ is entirely unclear, and we leave this as an open problem.
	\begin{prob}\label{prob:main}
		Solve the $K_r$ clique supersaturation problem for $K_{s,t}$ with $r,s,t\ge 3$.
	\end{prob}
	One can check that \Cref{thm:randomCliques} gives better bounds for $K_{s,t}$ compared to the bound \eqref{eq:Gnp} coming from $G_{n,p}$ if and only if
	\[r> \frac{2st-s-t}{s+t-2}=s+\frac{(s-1)(t-s)}{s+t-2}=2s-1-\frac{2(s-1)^2}{s+t-2}.\]
	In particular, \Cref{thm:randomCliques} never gives effective bounds when $r\le s$, and for $r\ge 2s-1$ it always gives a non-trivial bound.  \Cref{thm:randomCliques} used cliques placed uniformly at random, and one might hope that by placing cliques in a more careful way (say with the aid of a bipartite graph $B$ which avoids certain structures), one could obtain bounds better than \eqref{eq:Gnp} for smaller values of $r$.  Unfortunately, the following shows that this is essentially impossible.
	
	\begin{lem}\label{lem:KstFail}
		Let $2\le s\le t$ be integers and $r\le \frac{2st-s-t}{s+t-2}$.  There exists a constant $k_0=k_0(r,s,t)$ such that if $G$ is an $n$-vertex graph which is the union of $u$ cliques of size $m$ with $um^r=k n^{r-{r\choose 2}/s}$ where $k\ge k_0$, and with every edge contained in at most $O(1)$ of the $u$ cliques, then
		\[\N(K_{s,t},G)=\Om\left(k^{st/{r\choose 2}}n^{s}\right).\]
	\end{lem}
	We note that the quantity $um^r$ is roughly the number of copies of $K_r$ within one of the $u$ cliques making up $G$.
	
	\begin{proof}[Proof Sketch]
		If $m\ll k^{\frac{1}{r(r-1)}} n^{\frac{2s-r-1}{2s}}$, then a small computation together with the fact that each edge is in at most $O(1)$ cliques shows $e(G)\gg k^{1/{r\choose 2}} n^{2-1/s}$, from which the result follows by \Cref{prop:ESKst}.  Otherwise, counting copies of $K_{s,t}$ within each of the $u$ cliques gives
		\[\N(K_{s,t},G)=\Om(u m^{s+t})=\Om(kn^{r-{r\choose 2}/s}\cdot m^{s+t-r})=\Om(k^{st/{r\choose 2}}n^s),\]
		where this last step implicitly uses the upper bound on $r$.
	\end{proof}
	
	\Cref{lem:KstFail} shows that any non-trivial construction for \Cref{prob:main} when  $r\le s$ must look substantially different from all the constructions used throughout this paper.  We feel like such constructions should not exist, and as such we conjecture the following.

	\begin{conj}\label{conj:Kst}
		If $2\le r\le s\le t$, then there exists a constant $k_0$ such that if $G$ is an $n$-vertex graph with $\N(K_r,G)=k n^{r-{r\choose 2}/s}$ and $k\ge k_0$, then
		\[\N(K_{s,t},G)\ge k^{st/{r\choose 2}} n^{s-o(1)}.\]
	\end{conj}
	That is, if $r\le s$, then we predict every graph $G$ with a given number of $K_r$'s contains at least as many $K_{s,t}$'s as the random graph with the same number of $K_r$'s.  It is perhaps natural to extend this conjecture to all $r\le \frac{2s-s-t}{s+t-2}$ since this is the full range for \Cref{lem:KstFail}, but we find this quantity too strange to make any statements with confidence.
	
	Note that \Cref{prop:ESKst} implies \Cref{conj:Kst} for $r=2$.  Thus the next open case is $r=s=3$, which we restate below.
	\begin{conj}\label{conj:K3t}
		For all $t\ge 3$ there exists a constant $k_0$ such that if $G$ is an $n$-vertex graph with $k n^{2}$ triangles and $k\ge k_0$, then
		\[\N(K_{3,t},G)\ge k^{t} n^{3-o(1)}.\]
	\end{conj}
	We believe we can adapt the argument of \Cref{thm:graphK2t} for this problem to prove a lower bound of roughly $n^{7/3}$ when $k$ is a large constant, but it seems like new ideas are needed to improve upon this.

	\bibliographystyle{abbrv}
	\bibliography{refs}
\end{document}